\documentclass{amsart}
\usepackage[latin1]{inputenc}
\usepackage[T1]{fontenc}
\usepackage{lmodern}
\usepackage[english]{babel}
\usepackage{microtype}

\usepackage{amsmath,amssymb,amsfonts,amsthm}
\usepackage{mathtools,accents}
\usepackage{mathrsfs}
\usepackage{aliascnt}
\usepackage{braket}
\usepackage{bm}
\usepackage{lineno}
\usepackage[citecolor=blue,colorlinks]{hyperref}

\usepackage{enumerate}
\usepackage{xcolor}

\usepackage{aliascnt}

\makeatletter
\def\newaliasedtheorem#1[#2]#3{
  \newaliascnt{#1@alt}{#2}
  \newtheorem{#1}[#1@alt]{#3}
  \expandafter\newcommand\csname #1@altname\endcsname{#3}
}
\makeatother

\theoremstyle{plain}
\newtheorem{theorem}{Theorem}[section]
\newaliasedtheorem{lemma}[theorem]{Lemma}
\newaliasedtheorem{prop}[theorem]{Proposition}
\newaliasedtheorem{claim}[theorem]{Claim}
\newaliasedtheorem{corollary}[theorem]{Corollary}

\theoremstyle{definition}
\newaliasedtheorem{definition}[theorem]{Definition}
\newaliasedtheorem{example}[theorem]{Example}

\theoremstyle{remark}
\newaliasedtheorem{remark}[theorem]{Remark}
\numberwithin{equation}{section}

\def\11{{\rm 1~\hspace{-1.4ex}l} }
\def\R{\mathbb R}

\def\Z{\mathbb Z}

\def\E{\mathbb E}
\def\T{\mathbb T}

\title
[NLS with spatial white noise ]
{Global dynamics of the $2d$ NLS with white noise potential and generic polynomial nonlinearity}
\author[N. Tzvetkov \and N. Visciglia]
{N. Tzvetkov \and N. Visciglia}

\address{N.~Tzvetkov, CY Cergy-Paris Universit\'e,  Cergy-Pontoise, F-95000, UMR 8088 du CNRS}
\email{nikolay.tzvetkov@cyu.fr}

\address{N.~Visciglia, Dipartimento di Matematica, Universit\`a di Pisa, Largo Bruno Pontecorvo, 5, 56100 Pisa, Italy}
\email{nicola.visciglia@unipi.it}
\thanks{N.T. is supported by ANR grant ODA (ANR-18-CE40-0020-01), N.V. is supported by the project PRIN grant 2020XB3EFL and he acknowledge the Gruppo Nazionale per l' Analisi Matematica, la Probabilit\`a e le loro Applicazioni (GNAMPA) 
of the Istituzione Nazionale di Alta Matematica (INDAM)}

\begin{document}

\maketitle

\begin{abstract}
Using an approach introduced by Hairer-Labb\' e we construct a unique global dynamics for  the NLS on $\T^2$ with a white noise potential and an arbitrary polynomial nonlinearity. We build the solutions as a limit of classical solutions (up to a phase shift) of the same equation with smoothed potentials.
This is an improvement on previous contributions of us and  Debussche-Weber dealing with quartic nonlinearities and cubic nonlinearities respectively. 
\end{abstract}
\section{Introduction}
The aim of this work is to extend the result of  \cite{TV}  to an arbitrary polynomial nonlinearity. As announced in \cite{TV} this will require, in addition to the modified energies introduced in 
 \cite{TV}, a suitable use of the dispersive effect. 
 \\
 
We therefore aim to solve, in a sense to be defined, the following Cauchy problem 
\begin{equation}\label{NLSwhite}
i\partial_t u  = \Delta u + \xi u - u|u|^p, \quad
u(0,x)=u_0(x), \quad (t,x)\in \R\times \T^2,
\end{equation}
where $p\geq 2$ measures the strength of the nonlinear interaction and $\xi(x,\omega)$ is the (zero mean value) space white noise which can be seen as the distribution of the random Fourier  series 
$$
\xi(x,\omega)=\sum_{n\in\Z^2,n\neq 0}\, g_{n}(\omega)\,e^{in\cdot x}\,,
$$
where $g_n(\omega)$ are standard complex gaussians such that  $\overline{g_n}(\omega)=g_{-n}(\omega)$ and otherwise independent. 
\\

Thanks to the work by Bourgain \cite{Bo}, we know how to construct the global dynamics of \eqref{NLSwhite} if $\xi$ is replaced by a smooth potential. 
Therefore a natural way to solve  \eqref{NLSwhite} is to regularize  $\xi$ and to try to pass to a limit in the regularized problems. 
As shown in \cite{DW,TV} such a passage to limit is possible for $p\leq 3$ but only for well-prepared initial data.
Therefore, we are  interested in the solutions to the following regularization of \eqref{NLSwhite}
\begin{equation}\label{NLSwhite_tris}
i\partial_t u_\varepsilon  = \Delta u_\varepsilon + \xi_\varepsilon(x,\omega) u_\varepsilon - u_\varepsilon |u_\varepsilon|^p\,\, , \quad u_{\varepsilon}(0,x)=
u_0(x) e^{Y(x,\omega)-Y_\varepsilon(x,\omega)}\, \, ,
\end{equation}
where $\xi_\varepsilon=\chi_{\varepsilon}\ast \xi$, $\varepsilon \in (0,1)$ is a regularization of $\xi$ with
$\chi_{\varepsilon}(x)=\varepsilon^{-2}\chi(x/\varepsilon)$, where $\chi(x)$ is smooth with a support in $\{|x|<1/2\}$
and $\int_{\T^2}\chi dx=1$.  
As in \cite{DW,TV}, in \eqref{NLSwhite_tris},  $Y=\Delta^{-1} \xi$ and $Y_\varepsilon=\Delta^{-1} \xi_\varepsilon$ is its regularization .
\\

The main result of the paper is the following one, which is an extension of the one proved in \cite{TV} where we were restricted to the powers
$p\in [2,3]$.
\begin{theorem} \label{corproof}   Assume $p\geq 2$
and $u_0(x)$ be such that $e^{Y(x,\omega)} u_0(x) \in H^2(\T^2)$ a.s. 
Then there exists an event $\Sigma\subset \Omega$ such that $p(\Sigma)=1$ and
for every $\omega\in \Sigma$
there exists   
$$ 
v (t,x, \omega)\in \bigcap_{\gamma\in [0, 2)}{\mathcal C}(\R; H^\gamma(\T^2))
$$
such that for every $T>0$ and $\gamma \in [0,2)$ we have:
\begin{equation}\label{eas}
\sup_{t\in [-T,T]} \|e^{-iC_\varepsilon t} e^{Y_\varepsilon(x, \omega)} u_\varepsilon(t,x,\omega)- v (t,x,\omega)\|_{H^\gamma(\T^2)}
\overset{\varepsilon\rightarrow 0} \longrightarrow 0,
\end{equation} where 
$
C_\varepsilon=\E (|\nabla Y_\varepsilon(x,\omega)|^2)
$
(this quantity is independent of $x$)
and $u_\varepsilon(t,x,\omega)$ are solutions to \eqref{NLSwhite_tris}.
Moreover for $\gamma\in [0,1)$ and $\omega\in \Sigma$ we have
\begin{equation}\label{diffnew}
\sup_{t\in [-T,T]} \big \||u_\varepsilon(t,x,\omega)| - e^{-Y(x,\omega)}| v(t,x, \omega)| \big \|_
{{H^\gamma(\T^2)}\cap L^\infty (\T^2)} \overset{\varepsilon\rightarrow 0}
\longrightarrow 0.
\end{equation}
\end{theorem}
The proof of Theorem \ref{corproof}   crucially relies on the modified energies and some results from our previous paper \cite{TV}. This  makes that the present paper is not self-contained.  
As announced in  \cite{TV}, the new ingredient allowing to deal with general nonlinearities  is the use of the dispersive effect which leads to Proposition~\ref{maiM} below.  Proposition~\ref{maiM} displays a gain of regularity with respect to the Sobolev inequality used in \cite{TV}, once a time averaging is performed. Note that we allow logarithmic losses in $\varepsilon$ in our dispersive bounds. As already used in \cite{DW} these losses can be compensated by the polynomial in $\varepsilon$ convergence of $ Y_\varepsilon$ to $Y$ in the natural norms. 
\\

In \cite{GUZ,MZ} a different approach to the study of \eqref{NLSwhite} is introduced. This approach is based on the construction of a suitable  self adjoint realization of $ \Delta  + \xi$.
Then the initial data in \eqref{NLSwhite}  is chosen in the domain of this self adjoint operator. Being in the domain of this self adjoint operator is the substitute of our assumption of well prepared initial data  $e^{Y(x,\omega)} u_0(x) \in H^2(\T^2)$. At the best of our knowledge the present paper is the first one where 
global well-posedness is proved for \eqref{NLSwhite} with an arbitrary polynomial nonlinearity $p$, extending the papers \cite{DW} and \cite{TV}. Our proof is based on the approach
introduced by Hairer-Labb\' e in \cite{HL}.
\\

For the sake of simplicity, we prove Theorem~\ref{corproof} in the context of the flat  torus $\T^2$. 
However, it is quite likely that a similar result holds in the context of a general compact riemannian boundaryless manifolds.
Indeed, the dispersive estimates can be extended to this setting in a relatively straightforward way. The stochastic analysis results from \cite{TV} can also be extended to this setting by some slightly more involved elaborations. We will address this question and some related issues in a forthcoming work. 
\\

Following  Hairer-Labb\' e \cite{HL}, we set 
\begin{equation}\label{gaugetranf}
v_\varepsilon(t,x,\omega)= e^{-iC_\varepsilon t} e^{Y_\varepsilon(x,\omega)} u_\varepsilon (t,x,\omega)\, ,
\end{equation}
where $C_\varepsilon$ is the constant appearing in Theorem~\ref{corproof}.
Then $v_\varepsilon$ solves 
\begin{equation}\label{NLSepsilon}
i\partial_t v_\varepsilon  = \Delta v_\varepsilon  - 2\nabla v_\varepsilon  \cdot \nabla Y_\varepsilon (x,\omega)+ v_\varepsilon :|\nabla Y_\varepsilon |^2:(x,\omega)
-  e^{-pY_\varepsilon} v_\varepsilon|v_\varepsilon|^p,
\end{equation}
where 
\begin{equation*}
:|\nabla Y_\varepsilon |^2:(x,\omega)=|\nabla Y_\varepsilon |^2(x,\omega)-C_\varepsilon.
\end{equation*}
Following Section 6 in \cite{TV} the proof of
Theorem \ref{corproof} follows from the next theorem concerning the behavior of $v_\varepsilon(t,x,\omega)$,
where $:|\nabla Y |^2:(x,\omega)$ is the renormalized potential defined in \cite{TV}.
\begin{theorem}\label{main} Assume $p\geq 2$ and $u_0(x)$ be such that $e^{Y(x,\omega)} u_0(x) \in H^2(\T^2)$ a.s.
Then there exists an event $\Sigma\subset \Omega$ such that $p(\Sigma)=1$ and
for every $\omega\in \Sigma$  there exists $$ v(t,x,\omega)\in \bigcap_{\gamma\in [0, 2)}{\mathcal C}(\R; H^\gamma(\T^2))$$
such that  for every fixed $T>0$ and $\gamma\in [0,2)$ we have: 
$$\sup_{t\in [-T, T]} \|v_\varepsilon(t,x, \omega) - v(t,x, \omega)\|_{H^\gamma(\T^2)}
\overset{\varepsilon\rightarrow 0}\longrightarrow 0.$$
Here we have denoted by
$v_\varepsilon(t,x,\omega)$  for $\omega\in \Sigma$ the unique global solution in the space ${\mathcal C} (\R;H^2(\T^2))$ of the following problem:
\begin{multline}\label{NLSgaugeintronew}
i\partial_t v_\varepsilon  = \Delta v_\varepsilon  - 2\nabla v_\varepsilon  \cdot \nabla Y_\varepsilon(x,\omega) + v_\varepsilon :|\nabla Y_\varepsilon |^2:(x,\omega)
- e^{-pY_\varepsilon}v_\varepsilon|v_\varepsilon|^p, \\
v_\varepsilon(0,x)=v_0(x)\in H^2(\T^2)
\end{multline}
and $ v(t,x, \omega)$ denotes for $\omega\in \Sigma$ the unique global solution in the space ${\mathcal C} (\R; H^\gamma(\T^2))$, for $\gamma\in (1, 2)$,
of the following limit problem:
\begin{multline}\label{NLSgaugeintronewbar}
i\partial_t  v  = \Delta v  - 2\nabla  v  \cdot \nabla Y(x,\omega) +  v :|\nabla Y |^2:(x,\omega)
- e^{-pY} v| v|^p,\,\\
v(0,x)=v_0(x)\in H^2(\T^2)
\end{multline}
where in both Cauchy problems \eqref{NLSgaugeintronew}
and \eqref{NLSgaugeintronewbar} $v_0(x)=e^{Y(x,\omega)} u_0(x)$, $\omega\in \Sigma$.
\end{theorem}
\noindent{\bf Notations} For every $s\in \R$ we  denote $s^+$ any number belonging to $(s,s+\delta)$ for a suitable $\delta>0$, similarly 
$s^-$ denotes any number in $(s-\delta,s)$ for a suitable $\delta>0$.
We shall denote by $L^p, H^s, W^{s,p}$ the functional spaces $L^p(\T^2), H^s(\T^2), W^{s,p}(\T^2)$.
In the sequel we shall denote by $C$ any deterministic finite constant that can change from line to line and by $C(\omega)$ any random variable defined on $\Omega$ and finite a.s.
We shall denote by $C(\omega,T)$ a constant which is increasing w.r.t. $T$ and finite for every $(\omega, T)\in \Sigma\times \R^+$ for a suitable event $\Sigma\subset \Omega$ of full measure.
In the rest of the paper for shortness we will drop writing the $\omega$ dependence of $v_\varepsilon$ and $Y_\varepsilon$.
For every $a,b$ we denote by $\int_a^b$ the integral w.r.t. time variable and $\int_{\T^2}$ the integral on $\T^2$. 
\section{Preliminary facts}

We collect in this section some facts proved 
in \cite{TV} and some useful consequences that will be needed in the sequel.

\begin{prop}
We have the following bound:
\begin{equation}\label{Yest}
\sup_{\varepsilon\in (0,1)} \|Y_\varepsilon(x)\|_{L^\infty}\leq C(\omega) \, ,
\end{equation}
\begin{equation}\label{Yestnew}
\sup_{\varepsilon\in (0,1)} \|\nabla Y_\varepsilon(x)\|_{L^p}\leq C(\omega)|\log \varepsilon|, \quad \forall p\in [1, \infty)\, ,
\end{equation}
\begin{equation}\label{Yestnewnew}
\sup_{\varepsilon\in (0,1)} \|:|\nabla Y_\varepsilon|^2:(x)\|_{L^p}\leq C(\omega)|\log \varepsilon|^2\,,\quad \forall p\in [1, \infty)\, .
\end{equation}
For every $T>0$ we have the following estimates for the solutions $v_\varepsilon(t,x)$ of \eqref{NLSgaugeintronew}:
\begin{equation}\label{h1} \sup_{\varepsilon\in (0,1)} \|v_\varepsilon(t,x)\|_{L^\infty((0,T);H^1)}\leq C(\omega),
\end{equation}
\begin{equation}\label{h1+}
\sup_{\varepsilon\in (0,1)} 
\|v_\varepsilon(t,x)\|_{L^\infty((0,T);H^{1^+})}\leq C(\omega) \|v_\varepsilon\|_{L^\infty((0,T);H^2)}^{0^+}\, ,
\end{equation}
\begin{equation}\label{ellreg}
\sup_{\varepsilon\in (0,1)} 
\|v_\varepsilon(t,x)\|_{L^\infty ((0,T);H^2)} \leq C(\omega)+ C(\omega) \|e^{-Y_\varepsilon} \Delta v_\varepsilon\|_{L^\infty((0,T);L^2)}.
\end{equation}
\end{prop}
\begin{proof} The bounds \eqref{Yest}, \eqref{Yestnew}, 
\eqref{Yestnewnew} 
have been established in \cite{TV} as well as
\eqref{h1}. The estimate \eqref{h1+} follows by combining interpolation and \eqref{h1},
\eqref{ellreg} follows by combining elliptic regularity with \eqref{Yest}.
\end{proof}
Next we introduce the family of operators:
\begin{equation}\label{Hepsilon}
H_{\varepsilon} u=\Delta u  - 2 \nabla u\cdot \nabla Y_\varepsilon(x) +  u :|\nabla Y_\varepsilon |^2:(x),
\end{equation}
where as usual we drop the $\omega$ dependence of the operators $H_\varepsilon$.
In the sequel we shall need the following result.
\begin{prop}
We have the bound:
\begin{equation}\label{bound1} 
\|(H_\varepsilon -\Delta) u\|_{L^2}\leq C(\omega) |\ln \varepsilon|^C \| u\|_{H^{1^+}}.
\end{equation} 
\end{prop}
\begin{proof}
It is sufficient to show the bounds
\begin{equation}\label{inTE} \|\nabla u  \cdot \nabla Y_\varepsilon\|_{L^2}\leq C(\omega)\| u\|_{H^{1^+}}, \quad \|u :|\nabla Y_\varepsilon |^2:\|_{L^2} \leq C(\omega)\| u\|_{H^{1^+}}.\end{equation}
We have for every $\delta\in (0,1)$
$$\|\nabla u  \cdot \nabla Y_\varepsilon\|_{L^2}\leq C \|\nabla Y_\varepsilon\|_{L^{\frac 2\delta}} \|\nabla u\|_{L^{\frac 2{1-\delta}}}\leq C(\omega)|\ln \varepsilon| 
\|u\|_{H^{1+\delta}}
$$
where we have used \eqref{Yestnew} and the embedding $H^\delta\subset L^{\frac 2{1-\delta}}$.
The second bound in \eqref{inTE} follows by a similar argument
$$
\|u :|\nabla Y_\varepsilon |^2:\|_{L^2}\leq C \|:|\nabla Y_\varepsilon|^2:\|_{L^{\frac 2\delta}} \|u\|_{L^{\frac 2{1-\delta}}}\leq C(\omega) |\ln \varepsilon|^2 \|u\|_{H^{1+\delta}}
$$
where we have used \eqref{Yestnewnew} and $H^{1+\delta}\subset L^{\frac 2{1-\delta}}$.
\end{proof}
\section{
A priori bounds of $v_\varepsilon$}
We introduce the propagator $S_\varepsilon(t)$ associated with the linear problem
$
i\partial_t u = H_{\varepsilon} u,
$
where $H_\varepsilon$ is defined in \eqref{Hepsilon}.
The main point of this section is  Proposition \ref{BF}.  In order to prove it, we shall need Strichartz estimates with loss
for the propagator $S_\varepsilon(t)$.
\begin{prop}\label{strich}
For every $T>0$ we have the following bound:
\begin{equation}\label{energy}
\|S_\varepsilon (t) \varphi\|_{L^\infty((0,T);H^s)}\leq C(\omega) |\log \varepsilon |^C \|\varphi\|_{H^s}, \quad s\in [0,2].
\end{equation}
Moreover for every $r,q\in (2,\infty)$ such that $\frac 2r+\frac 2q=1$
we have \begin{equation}\label{strich+}
\|S_\varepsilon(t)\varphi\|_{L^r((0,T);L^q)} \leq C(\omega,T) |\log \varepsilon |^{C} \|\varphi\|_{H^{\frac {1}{r}^+}}.
\end{equation}
\end{prop}

\begin{proof}
Estimate \eqref{energy} is established in \cite{DW}. For the proof of \eqref{strich+}, we follow the argument of \cite{MZ} which is closely related to the analysis in \cite{BGT,KT,ST,T}.
The basic strategy is to perform a perturbative argument with respect to the evolution $\exp(it\Delta)$ by a partition on small time intervals which makes the perturbation 
$H_\varepsilon - \Delta$ better but which losses some regularity on the data because of the summation on the small time intervals. 
An additional  difficulty resolved in \cite{MZ} is coming from the fact that a frequency localisation of  $(H_\varepsilon- \Delta)(u)$  does not imply a frequency localisation of $u$.
\\

Let 
$$
{\rm Id}=\sum_{N-{\rm dyadic}}\Delta_N
$$
be a Littlewood-Paley partition of the unity.  Therefore the issue is to bound 
\begin{equation}\label{object}
\|\Delta_{N_1}S_\varepsilon(t)\Delta_{N_2}\varphi\|_{L^r( (0,T)  ;L^q )}\, .
\end{equation}
In order to evaluate \eqref{object}, we distinguish two cases according to the sizes of $N_1$ and $N_2$
and to sum up on $N_1, N_2$.
\\
\\
{\em First case: $N_1\geq N_2$}\\
\\
In this case we split the interval $[0,T]$ in an essentially disjoint union of intervals of size $ N_1^{-1}$ as
\begin{equation}\label{spli}
[0,T]=\bigcup_j I_j
\end{equation}
and we aim to estimate 
$
\|\Delta_{N_1}S_\varepsilon(t)\Delta_{N_2}\varphi\|_{L^r(I_j;L^q )}\, .
$
Suppose that $I_j=[a,b]$. Then following \cite{MZ} (see also \cite{JSS}), for $t\in [a,b]$ we can write 
\begin{multline}\label{decomposition}
\Delta_{N_1}S_\varepsilon(t)\Delta_{N_2}\varphi
=
\Delta_{N_1} e^{i(t-a)\Delta} S_{\varepsilon}(a)\Delta_{N_2}\varphi
\\
+
i\int_a^t \Delta_{N_1} e^{i(t-\tau)\Delta} (H_\varepsilon -\Delta) S_{\varepsilon}(\tau) \Delta_{N_2}\varphi d\tau.
\end{multline}
We now estimate each term in the right hand-side of \eqref{decomposition}.  Using \cite{BGT}, we estimate the first term as follows for $\delta>0$:
\begin{multline*}
\|\Delta_{N_1} e^{i(t-a)\Delta} S_{\varepsilon}(a)\Delta_{N_2}\varphi\|_{L^r(I_j;L^q)}
\leq CN_1^{-\frac{1}{r}-\delta}
\|S_{\varepsilon}(a)\Delta_{N_2}\varphi\|_{H^{\frac{1}{r}+\delta}}\\\leq 
 C(\omega) |\log \varepsilon |^C N_1^{-\frac{1}{r}-\delta} \|\varphi\|_{H^{\frac{1}{r}+\delta}}
\end{multline*}
where we have used \eqref{energy}.
Now we estimate the second term in the right hand-side of \eqref{decomposition}.
Using the Minkowski inequality and  \cite{BGT}, we can write for every $\delta>0$:
\begin{multline*}
\Big \|\int_a^t \Delta_{N_1} e^{i(t-\tau)\Delta} (H_\varepsilon -\Delta) S_{\varepsilon}(\tau) \Delta_{N_2}\varphi d\tau \Big \|_{L^r(I_j;L^q)}
\\
\leq C
\int_{I_j}
\|
(H_\varepsilon -\Delta) S_{\varepsilon}(\tau) \Delta_{N_2}\varphi\|_{L^2} d\tau\\
\leq 
C(\omega) |\log \varepsilon |^C
N_1^{-1}
N_2^{1+ \frac\delta 2 }
N_2^{-\frac{1}{r}-\delta}
 \|\varphi\|_{H^{\frac{1}{r}+\delta}}
\end{multline*}
where we have used \eqref{bound1} and  \eqref{energy}.
Summarizing we get
$$\|\Delta_{N_1}S_\varepsilon(t)\Delta_{N_2}\varphi\|_{L^r(I_j;L^q)}\leq 
C(\omega) |\log \varepsilon |^C \big( N_1^{-\frac{1}{r}-\delta} +N_1^{-1}
N_2^{1-\frac 1r-\frac \delta 2}\big )
 \|\varphi\|_{H^{\frac{1}{r}+\delta}} 
$$
and hence using that the number of $I_j$ is smaller than $TN_1$ taking the $r$'th power of the previous bound and summing on $j$, we get the estimate 
$$\|\Delta_{N_1}S_\varepsilon(t)\Delta_{N_2}\varphi\|_{L^r((0,T);L^q)}\leq C(\omega) T^\frac 1r |\log \varepsilon |^C \big( N_1^{-\delta} +N_1^{-1+\frac 1r}
N_2^{1-\frac 1r-\frac \delta 2}\big )
 \|\varphi\|_{H^{\frac{1}{r}+\delta}}$$
and hence 
\begin{equation}\label{111}
\sum_{N_2\leq N_1} \|\Delta_{N_1}S_\varepsilon(t)\Delta_{N_2}\varphi\|_{L^r((0,T);L^q)}\leq C(\omega) T^\frac 1r|\log \varepsilon |^C  \|\varphi\|_{H^{\frac{1}{r}+\delta}}
\end{equation}
where we have used
$$\sum_{N_2\leq N_1} \big( N_1^{-\delta} +N_1^{-1+\frac 1r}
N_2^{1-\frac 1r-\frac \delta 2}\big )
<\infty.$$
{\em Second case: $N_1\leq N_2$}\\
\\
 We consider again the splitting \eqref{spli} but this time the intervals $I_j$ are of size $ N_2^{-1}$.
Again we consider \eqref{decomposition} and we estimate each term of the right hand-side. 
Since $ N_2^{-1}\leq N_1^{-1}$,  using \cite{BGT} and \eqref{energy}, we estimate the first term at the  right hand-side of \eqref{decomposition} as 
$$
\|\Delta_{N_1} e^{i(t-a)\Delta} S_{\varepsilon}(a)\Delta_{N_2}\varphi\|_{L^r(I_j;L^q)}
\leq  C(\omega) |\log \varepsilon |^CN_2^{-\frac{1}{r}-\delta}  \|\varphi\|_{H^{\frac{1}{r}+\delta}}\, ,
$$
where $\delta>0$.
Next,  as above, we can estimate the second term at the  right hand-side of \eqref{decomposition} as 
\begin{multline*}
\Big \|\int_a^t \Delta_{N_1} e^{i(t-\tau)\Delta} (H_\varepsilon -\Delta) S_{\varepsilon}(\tau) \Delta_{N_2}\varphi d\tau\Big \|_{L^r(I_j;L^q)}
\\
\leq 
C(\omega) |\log \varepsilon |^C
N_2^{-1}
N_2^{1+\frac \delta 2}
N_2^{-\frac{1}{r}-\delta}
 \|\varphi\|_{H^{\frac{1}{r}+\delta}}\, .
\end{multline*}
Summarizing we get
$$\|\Delta_{N_1}S_\varepsilon(t)\Delta_{N_2}\varphi\|_{L^r(I_j;L^q)}\leq 
C(\omega) |\log \varepsilon |^C \big( N_2^{-\frac{1}{r}-\delta} + 
N_2^{-\frac{1}{r}-\frac \delta 2} \big)  \|\varphi\|_{H^{\frac{1}{r}+\delta}}$$
and as above, using that the number of $I_j$ is smaller than $TN_2$ taking the $r$'th power of the previous bound and summing on $j$, we get the estimate 
$$
\|\Delta_{N_1}S_\varepsilon(t)\Delta_{N_2}\varphi\|_{L^r((0,T);L^q)}
\leq 
C(\omega) T^\frac 1r |\log \varepsilon |^C \big( N_2^{-\delta} +
N_2^{-\frac \delta 2} \big)
\|\varphi\|_{H^{\frac{1}{r}+\delta}}.
$$
Hence we get
\begin{equation}\label{222}\sum_{N_1\leq N_2} \|\Delta_{N_1}S_\varepsilon(t)\Delta_{N_2}\varphi\|_{L^r((0,T);L^q)}
\leq 
C(\omega) T^\frac 1r |\log \varepsilon |^C \|\varphi\|_{H^{\frac{1}{r}+\delta}}
\end{equation}
since $$\sum_{N_1\leq N_2} \big( N_2^{-\delta} +
N_2^{-\frac \delta 2} \big)
<\infty.$$
We conclude by combining \eqref{111} and \eqref{222} with the Minkowski inequality.
\end{proof}
As a consequence we get the following result.
\begin{prop}\label{strichcor}
For every $T>0$ we have the following estimates:
\begin{equation}\label{noDuhamel}
\|S_\varepsilon(t)\varphi\|_{L^4((0,T);W^{\frac 34^-,4})}\leq C(\omega,T) |\log \varepsilon|^C \|\varphi\|_{H^1}
\end{equation}
and
\begin{equation}\label{Duhamel}
\Big \|\int_0^t S_\varepsilon(t-s) f(s) ds\Big \|_{L^4((0,T);W^{\frac 34^-,4})}\leq C(\omega,T)  |\log \varepsilon|^C \|f\|_{L^1((0,T);H^1)}\, .
\end{equation}
\end{prop}
\begin{proof}
Notice that \eqref{Duhamel} follows by combining \eqref{noDuhamel} with the Minkowski inequality.
Next we focus on the proof of \eqref{noDuhamel}.
Notice that for every $\varepsilon_0\in (0,1)$, there exists $q\in (1, \infty)$ such that 
the following Gagliardo-Nirenberg inequality occurs:
$$\|u\|_{W^{\frac 34-\varepsilon_0,4}}\leq C \|u\|_{L^q}^{\frac 12} \|u\|_{H^\frac 32}^{\frac 12}$$
and hence by integration in time and H\"older inequality in time we get
\begin{multline*}\|S_\varepsilon(t)\varphi\|_{L^4((0,T);W^{\frac 34-\varepsilon_0,4})}^4\leq C \|S_\varepsilon(t)\varphi\|_{L^{2}((0,T);L^q)}^{2} \|S_\varepsilon(t)\varphi\|_{L^\infty((0,T);H^\frac 32)}^{2}
\\\leq C(\omega) |\log \varepsilon|^C
 \|S_\varepsilon(t)\varphi\|_{L^{r}((0,T);L^q)}^{2} \|\varphi\|_{H^\frac 32}^{2}\leq C(\omega,T)|\log \varepsilon|^C \|\varphi\|_{H^{\frac 12^-}}^{2}
\|\varphi\|_{H^\frac 32}^{2}\end{multline*}
where $q,r$ are Strichartz admissible and we have used  \eqref{energy}, \eqref{strich+}.

Notice that for initial datum $\varphi=\Delta_N \varphi$ which is spectrally localize at dyadic frequency $N$ we get from the previous bound
$$\|S_\varepsilon(t)\Delta_N \varphi\|_{L^4((0,T);W^{\frac 34-\varepsilon_0,4})} \leq C(\omega,T) |\log \varepsilon|^C \|\Delta_N \varphi\|_{H^{1^-}}.
$$
We conclude \eqref{noDuhamel} by summing on $N$.
\end{proof}

As a consequence we get the following result.
\begin{prop}\label{strichcor}
For every $T>0$ we have the following estimates:
\begin{equation}\label{noDuhamel}
\|S_\varepsilon(t)\varphi\|_{L^4((0,T);W^{\frac 34^-,4})}\leq C(\omega,T) |\log \varepsilon|^C \|\varphi\|_{H^1}
\end{equation}
and
\begin{equation}\label{Duhamel}
\Big \|\int_0^t S_\varepsilon(t-s) f(s) ds\Big \|_{L^4((0,T);W^{\frac 34^-,4})}\leq C(\omega,T)  |\log \varepsilon|^C \|f\|_{L^1((0,T);H^1)}\, .
\end{equation}
\end{prop}
\begin{proof}
Notice that \eqref{Duhamel} follows by combining \eqref{noDuhamel} with the Minkowski inequality.
Next we focus on the proof of \eqref{noDuhamel}.
Notice that for every $\varepsilon_0\in (0,1)$, there exists $q\in (1, \infty)$ such that 
the following Gagliardo-Nirenberg inequality occurs:
$$\|u\|_{W^{\frac 34-\varepsilon_0,4}}\leq C \|u\|_{L^q}^{\frac 12} \|u\|_{H^\frac 32}^{\frac 12}$$
and hence by integration in time and H\"older inequality in time we get
\begin{multline*}\|S_\varepsilon(t)\varphi\|_{L^4((0,T);W^{\frac 34-\varepsilon_0,4})}^4\leq C \|S_\varepsilon(t)\varphi\|_{L^{2}((0,T);L^q)}^{2} \|S_\varepsilon(t)\varphi\|_{L^\infty((0,T);H^\frac 32)}^{2}
\\\leq C(\omega) |\log \varepsilon|^C
 \|S_\varepsilon(t)\varphi\|_{L^{r}((0,T);L^q)}^{2} \|\varphi\|_{H^\frac 32}^{2}\leq C(\omega,T)|\log \varepsilon|^C \|\varphi\|_{H^{\frac 12^-}}^{2}
\|\varphi\|_{H^\frac 32}^{2}\end{multline*}
where $q,r$ are Strichartz admissible and we have used  \eqref{energy}, \eqref{strich+}.

Notice that for initial datum $\varphi=\Delta_N \varphi$ which is spectrally localize at dyadic frequency $N$ we get from the previous bound
$$\|S_\varepsilon(t)\Delta_N \varphi\|_{L^4((0,T);W^{\frac 34-\varepsilon_0,4})} \leq C(\omega,T) |\log \varepsilon|^C \|\Delta_N \varphi\|_{H^{1^-}}.
$$
We conclude \eqref{noDuhamel} by summing on $N$.
\end{proof}
Next we  get the following bound on the nonlinear solutions $v_\varepsilon$ to \eqref{NLSgaugeintronew}.
\begin{prop}\label{BF}
For every $T>0$ we have the following bound:
\begin{equation}
\|v_\varepsilon(t,x)\|_{L^4((0,T);W^{\frac 34^-,4})}\leq C(\omega,T) |\log \varepsilon|^C (1+ \|v_\varepsilon(t,x)\|_{L^\infty((0,T);H^2)}^{0^+}).
\end{equation}
\end{prop}
\begin{proof}
By combining Proposition \ref{strichcor} with the integral formulation associated with \eqref{NLSgaugeintronew}
we get:
\begin{multline*}\|v_\varepsilon\|_{L^4((0,T);W^{\frac 34^-,4})}\leq C(\omega,T)|\log \varepsilon|^C \|v_\varepsilon(0)\|_{H^1} 
+ C(\omega,T) |\log \varepsilon|^C  \int_0^T \|e^{-pY_\varepsilon} v_\varepsilon |v_\varepsilon|^p\|_{H^1}\\
\leq C(\omega,T) |\log \varepsilon|^C \|v_\varepsilon(0)\|_{H^1} 
+  C(\omega,T) |\log \varepsilon|^C  \int_0^T \|v_\varepsilon\|_{H^1} \|v_\varepsilon\|_{L^\infty}^p \|e^{-pY_\varepsilon}\|_{L^\infty}
\\+ C(\omega,T)|\log \varepsilon|^C  \int_0^T \|\nabla Y_\varepsilon\|_{L^2}  \|e^{-pY_\varepsilon}\|_{L^\infty} \|v_\varepsilon\|_{L^\infty}^{p+1}
\end{multline*}
and we conclude by using the Sobolev embedding $H^{1^+}\subset L^\infty$, \eqref{Yest}, \eqref{Yestnew}, \eqref{h1}, \eqref{h1+}.
\end{proof}
We conclude this section with the following key estimate.
\begin{prop}\label{maiM}
We have the following bound for a suitable $\eta\in(0,1)$ and for every $T>0$:
$$\|v_\varepsilon(t,x)\|_{L^2((0,T);W^{1,4})}^2\leq C(\omega,T) |\log \varepsilon|^C (1+\|v_\varepsilon(t,x)\|_{L^\infty((0,T);H^2)}^\eta).$$
\end{prop}
\begin{proof}
We have the bound for time independent functions:
$$\|u\|_{W^{1,4}}\leq C \|u\|_{W^{\frac 34^-, 4}}^{\frac 23^-} \|u\|_{H^2}^{\frac 13^+}.$$
Hence by integration in time and by choosing $u=v_\varepsilon$ we get
$$\|v_\varepsilon\|_{L^2((0,T);W^{1,4})}^2\leq C T\|v_\varepsilon\|_{L^{\frac 43^-}((0,T);W^{\frac 34^-, 4})}^{\frac 43^-} \|v_\varepsilon\|_{L^\infty((0,T);H^2)}^{\frac 23^+}
$$$$
\leq C T
\|v_\varepsilon\|_{L^4((0,T);W^{\frac 34^-, 4})}^{\frac 43^-} \|v_\varepsilon\|_{L^\infty((0,T);H^2)}^{\frac 23^+}.
$$
We conclude by Proposition \ref{BF}.
\end{proof}
\section{Proof of Theorem \ref{main}}
We aim at proving the following bound for every given $T>0$:
\begin{equation}\label{key}\|v_\varepsilon(t,x)\|_{L^\infty ((0,T);H^2)}\leq |\log \varepsilon|^{C(\omega,T)},\,\,  \forall \varepsilon\in(0,\frac 12).
\end{equation}
Recall that the bound \eqref{key}  has been achieved in \cite{TV} in the case $2\leq p\leq 3$ (see Proposition 4.5 in \cite{TV}). 
The main point is that we get the bound   \eqref{key} for every $p\geq 2$. Once \eqref{key} is achieved then Theorem \ref{main} can be proved exactly as in \cite[Section 5]{TV}.
\\

We can now establish \eqref{key}. In order to do that we recall some notations from \cite{TV}.
Denote by ${\mathcal H}_\varepsilon, {\mathcal F}_\varepsilon$ and ${\mathcal G}_\varepsilon$ the energies introduced along \cite[Proposition 4.1]{TV} which satisfy 
\begin{equation}\label{iaae}
\frac d{dt} ({\mathcal F}_\varepsilon(v_\varepsilon)-{\mathcal G}_\varepsilon(v_\varepsilon))=-{\mathcal H}_\varepsilon (v_\varepsilon).
\end{equation}
An important point is to obtain  the following modification of \cite[Proposition 4.3]{TV} which gains on the power of  $\|e^{-Y_\varepsilon} \Delta v_\varepsilon\|_{L^\infty((0,T); L^2)}$ 
appearing in the right hand-side by exploiting the averaging in the time variable.
\begin{prop}\label{firstcor} For a suitable $\gamma\in (1,2)$ we have the bound:
\begin{equation*}
\int_0^T |{\mathcal H}_\varepsilon(v_\varepsilon(s))|ds\leq C(\omega,T)|\log \varepsilon|^C + \|e^{-Y_\varepsilon} \Delta v_\varepsilon\|_{L^\infty((0,T); L^2)}^\gamma \, .
\end{equation*}
\end{prop}
\begin{proof}
By using the H\"older inequality, the Leibnitz rule and the diamagnetic inequality $|\partial_t |u||\leq |\partial_t u|$ we get that the first three terms in ${\mathcal H}_\varepsilon(v_\varepsilon)$ can be estimated by:
\begin{equation*} \int_{\T^2} |\partial_t v_\varepsilon ||\nabla v_\varepsilon|^2 |v_\varepsilon|^{p-1} 
e^{-(p+2)Y_\varepsilon}
\leq C(\omega) \|\partial_t v_\varepsilon \|_{L^2} 
\|\nabla v_\varepsilon\|_{L^4}^2 \|v_\varepsilon\|_{L^\infty}^{p-1}.
\end{equation*}
where we have used \eqref{Yest}.
By using the equation solved by $v_\varepsilon(t,x)$ and the Sobolev embedding $H^{1^+}\subset L^\infty$ we get 
from the estimate above after integration in time:
\begin{multline*}\int_0^T \int_{\T^2} |\partial_t v_\varepsilon ||\nabla v_\varepsilon|^2 |v_\varepsilon|^{p-1} 
e^{-(p+2)Y_\varepsilon}\\
\leq C(\omega)
\|\Delta v_\varepsilon \|_{L^\infty((0,T);L^2)} \|\nabla v_\varepsilon\|_{L^2((0,T);L^4)}^2
\|v_\varepsilon\|_{L^\infty((0,T); H^{1^+})}^{p-1}\\
+C(\omega)\|\nabla v_\varepsilon  \cdot \nabla Y_\varepsilon \|_{L^\infty((0,T);L^2)} 
\|\nabla v_\varepsilon\|_{L^2((0,T);L^4) }^2
\|v_\varepsilon\|_{L^\infty((0,T); H^{1^+})}^{p-1}
\\
+C(\omega)\|v_\varepsilon :|\nabla Y_\varepsilon |^2: \|_{L^\infty((0,T);L^2)} \|\nabla v_\varepsilon\|_{L^2((0,T);L^4)}^2
\|v_\varepsilon\|_{L^\infty((0,T); H^{1^+})}^{p-1}\\
+C(\omega)\|e^{-pY_\varepsilon}v_\varepsilon |v_\varepsilon|^p\|_{L^\infty((0,T);L^2)} \|\nabla v_\varepsilon\|_{L^2((0,T);L^4)}^2
\|v_\varepsilon\|_{L^\infty((0,T); H^{1^+})}^{p-1}
\\
=I+II+III+IV.\end{multline*}
Combining \eqref{Yest}, \eqref{h1+}, \eqref{ellreg} and Proposition \ref{maiM} we get
\begin{multline*}I\leq C(\omega,T) |\log \varepsilon|^C \|\Delta v_\varepsilon \|_{L^\infty((0,T);L^2)} \|v_\varepsilon\|_{L^\infty((0,T); H^{2})}^{0^+} 
\big (1+\|v_\varepsilon\|_{L^\infty((0,T);H^2)}^\eta\big)
\\\leq C(\omega,T) |\log \varepsilon|^C+ \|e^{-Y_\varepsilon} \Delta v_\varepsilon \|_{L^\infty((0,T);L^2)}^{1+\eta^+}.
\end{multline*}
By combining now H\"older inequality, \eqref{h1+} and Proposition \ref{maiM} we get
\begin{equation*}
II\leq C(\omega,T) |\log \varepsilon|^C\|\nabla Y_\varepsilon\|_{L^4} \|\nabla v_\varepsilon\|_{L^\infty((0,T);L^4)}  
\big(1+\|v_\varepsilon\|_{L^\infty((0,T);H^2)}\big) ^{\eta^+}
\end{equation*}
and hence by \eqref{Yestnew} and Sobolev embedding $H^1\subset L^4$ we conclude
\begin{multline*}II \leq C(\omega,T)
|\log \varepsilon|^C (1+\|v_\varepsilon\|_{L^\infty((0,T);H^2)})^{1+\eta^+} \\
\leq C(\omega,T) |\log \varepsilon|^C + \|e^{-Y_\varepsilon} \Delta v_\varepsilon \|_{L^\infty((0,T);L^2)}^{1+\eta^+}
\end{multline*}
where we used at the last step \eqref{ellreg}.
We also get
$$III \leq C(\omega,T) |\log \varepsilon|^C + \|e^{-Y_\varepsilon} \Delta v_\varepsilon \|_{L^\infty((0,T);L^2)}^{1+\eta^+}$$
whose proof is identical to the estimate of the term $II$ given above, except that we use \eqref{Yestnewnew} instead of \eqref{Yestnew}.
For the term $IV$ we get by \eqref{Yest}, \eqref{h1}, \eqref{h1+}, \eqref{ellreg} and Proposition \ref{maiM}
\begin{multline*}IV\leq C(\omega,T) |\log \varepsilon|^C \|v_\varepsilon\|_{L^\infty((0,T);L^{2(p+1)})}^{p+1}
 \big( 1+ \|v_\varepsilon\|_{L^\infty((0,T);H^2)})^{\eta^+}\\
\leq C(\omega,T) |\log \varepsilon|^C + \|e^{-Y_\varepsilon} \Delta v_\varepsilon \|_{L^\infty((0,T);L^2)}^{\eta^+},
\end{multline*}
where we used at the last step the Sobolev embedding $H^1\subset L^{2(p+1)}$.
Concerning 
the last term in the expression of ${\mathcal H}_\varepsilon (v_\varepsilon)$ we can estimate it as follows:
\begin{multline*}
\int_{\T^2} |\partial_t v_\varepsilon | |v_\varepsilon |^p 
|\nabla Y_\varepsilon| |\nabla v_\varepsilon| e^{-(p+2)Y_\varepsilon}\\
\leq C(\omega) \|v_\varepsilon \|_{L^{8p}}^p
\|\partial_t v_\varepsilon \|_{L^2} \|\nabla Y_\varepsilon\|_{L^8} 
 \|\nabla  v_\varepsilon\|_{L^4} 
\\
\leq C(\omega) |\log \varepsilon| \|\partial_t v_\varepsilon \|_{L^2} 
 \|\nabla v_\varepsilon \|_{L^4} 
\end{multline*}
where we have used \eqref{Yest}, \eqref{Yestnew}, the Sobolev embedding $H^1\subset L^{8p}$ and \eqref{h1}.
Next we replace $\partial_t v_\varepsilon$ by using the equation solved by $v_\varepsilon$ and,
thanks to the following time-independent Gagliardo-Nirenberg inequality
\begin{equation}\label{GN}
\|\nabla u\|_{L^4}^2\leq C \|\nabla u\|_{L^2}\|\Delta u\|_{L^2},\end{equation}
we can continue the estimate above as follows:
\begin{multline*}\dots \leq  C(\omega) |\log \varepsilon| \|\Delta v_\varepsilon \|_{L^2} 
\|\Delta v_\varepsilon\|_{L^2}^\frac 12
\|\nabla v_\varepsilon\|_{L^2}^\frac 12\\
+C(\omega) |\log \varepsilon| 
\|\nabla v_\varepsilon  \cdot \nabla Y_\varepsilon \|_{L^2} \|\Delta v_\varepsilon\|_{L^2}^\frac 12 
\|\nabla v_\varepsilon\|_{L^2}^\frac 12\\
+C(\omega) |\log \varepsilon| \|v_\varepsilon :|\nabla Y_\varepsilon |^2: \|_{L^2}
\|\Delta v_\varepsilon\|_{L^2}^\frac 12
\|\nabla v_\varepsilon\|_{L^2}^\frac 12
\\+C(\omega) |\log \varepsilon| \|e^{-pY_\varepsilon}v_\varepsilon |v_\varepsilon|^p\|_{L^2}
\|\Delta v_\varepsilon\|_{L^2}^\frac 12
\|\nabla v_\varepsilon\|_{L^2}^\frac 12
\end{multline*}
and by the Sobolev embedding $H^1\subset L^4$ and \eqref{Yest}, \eqref{h1}
\begin{multline*}\dots \leq C(\omega) |\log \varepsilon| 
\|\Delta v_\varepsilon\|_{L^2}^\frac 32 
+C(\omega) |\log \varepsilon| 
\|\nabla v_\varepsilon\|_{L^4} \|\nabla Y_\varepsilon \|_{L^4} \|\Delta v_\varepsilon \|_{L^2}^\frac 12 \\
+C(\omega) |\log \varepsilon| \|v_\varepsilon\|_{L^4} \|:|\nabla Y_\varepsilon |^2: \|_{L^4}
\|\Delta v_\varepsilon\|_{L^2}^\frac 12
+C(\omega) |\log \varepsilon| \|v_\varepsilon |v_\varepsilon|^p\|_{L^2}
\|\Delta v_\varepsilon\|_{L^2}^\frac 12
\\
\leq  C(\omega) |\log \varepsilon| 
\|\Delta v_\varepsilon\|_{L^2}^\frac 32 
+C(\omega) |\log \varepsilon|^2  \|\Delta v_\varepsilon \|_{L^2}^\frac 32
+C(\omega) |\log \varepsilon|^3 
\|\Delta v_\varepsilon\|_{L^2}^\frac 12
+C(\omega) |\log \varepsilon| 
\|\Delta v_\varepsilon\|_{L^2}^\frac 12
\end{multline*}
where we have used \eqref{Yestnew} and \eqref{Yestnewnew}.
Summarizing we get from the computation above and by \eqref{ellreg}
$$\int_0^T \int_{\T^2} |\partial_t v_\varepsilon | |v_\varepsilon |^p 
|\nabla Y_\varepsilon| |\nabla v_\varepsilon| e^{-(p+2)Y_\varepsilon}
\leq C(\omega,T) |\log \varepsilon|^C + \|e^{-Y_\varepsilon}\Delta v_\varepsilon \|_{L^2}^{\frac 32^+}.
$$
\end{proof}
Next we shall also need the following bound from \cite[Proposition 4.4]{TV}.
\begin{prop}\label{bounds} For every $\mu>0$ there exists a random variable $C(\omega)$ such that:
\begin{equation}\label{mart}
\big |{\mathcal F}_{\varepsilon}(v_\varepsilon)-\int_{\T^2} |\Delta v_\varepsilon|^2 e^ {-2Y_\varepsilon}
\big |< \mu \|e^{-Y_\varepsilon} \Delta v_\varepsilon\|_{L^2}^2
+
C(\omega) |\log \varepsilon|^4
\end{equation}
and
\begin{equation}\label{pastor}|{\mathcal G}_{\varepsilon}(v_\varepsilon)|<
\mu \|e^{-Y_\varepsilon} \Delta v_\varepsilon\|_{L^2}^2 +C(\omega)  |\log \varepsilon|^4.
\end{equation}
\end{prop}
We have now all tools to prove \eqref{key}.
By integration in time of \eqref{iaae} and by combining Proposition \ref{bounds} (where we choose $\mu$ small enough 
in order to absorb on the l.h.s. the term $\|e^{-Y_\varepsilon} \Delta v_\varepsilon\|_{L^\infty((0,T);L^2)}^2 )$ with Proposition \ref{firstcor}
we get
$$\|e^{-Y_\varepsilon} \Delta v_\varepsilon\|_{L^\infty((0,T);L^2)}^2\leq C(\omega, T) |\log \varepsilon|^C + \|e^{-Y_\varepsilon} \Delta v_\varepsilon\|_{L^\infty((0,T);L^2)}^\gamma,
\quad \gamma<2
$$
and hence we conclude \eqref{key}.


\begin{thebibliography}{10}

\bibitem{Bo} J. Bourgain,  
{\em Fourier transform restriction phenomena for certain lattice subsets and applications to nonlinear evolution equation}, Geom. and Funct. Anal. 3 (1993) 107--156, 209-262.

\bibitem{BGT} N.~Burq,  P.~G\'erard,  N.~Tzvetkov, {\it Strichartz inequalities and the nonlinear Schrodinger equation on compact manifolds}, Amer. J. Math. 126 (2004), 569--605. 

\bibitem{DW} A.~Debussche, H.~Weber, {\it The {S}chr\"{o}dinger equation with spatial white noise potential}, Electron. J. Probab., 23 (2018) no. 28, 16 pp.

 
\bibitem{GUZ} M. Gubinelli, B. Ugurcan, I. Zachhuber, {\it Semilinear evolution equations for the {A}nderson {H}amiltonian in two and three dimensions}, Stoch. Partial Differ. Equ. Anal. Comput., 8 (2020) 1, 82--149.

\bibitem{JSS}  J.L. Journ\'e, A. Soffer, C. Sogge, {\it Decay estimates for Schr\"dinger operators}  Comm. Pure Appl. Math. 44 (1991) 573--604.

\bibitem{HL} M. Hairer, C. Labb\' e, {\em A simple construction of the continuum parabolic {A}nderson model on {${\bf R}^2$}}, Electron. Commun. Probab., 20 (2015) no. 43, 11 pp.


\bibitem{KT} H. Koch,  N. Tzvetkov,  {\it On the local well-posednes of the Benjamin-Ono equation in $H^s$}, Int. Math. Res. Not.   26 (2003) 1449--1464. 

\bibitem{MZ} A. Mouzard, I. Zachhuber,  {\it Strichartz inequalities with white noise potential on compact surfaces}, arXiv:2104.07940 [math.AP]

\bibitem{ST} G. Staffilani, D. Tataru, {\it  Strichartz estimates for a Schr\"odinger operator with nonsmooth coefficients}, Comm. Partial Differential Equations 27 (2002)1337--1372. 

\bibitem{TV} N. Tzvetkov, N. Visciglia, {\it Two dimensional nonlinear Schr\"odinger equation with spatial white noise potential and fourth order nonlinearity}, arXiv:2006.07957 [math.AP]
accepted on Stochastics and Partial Differential Equations: Analysis and Computations
 
\bibitem{T} D.  Tataru, {\it  Strichartz estimates for operators with nonsmooth coefficients and the nonlinear wave equation}, Amer. J. Math. 122 (2000) 349--376. 


\end{thebibliography}
\end{document}